\documentclass[12pt,letterpaper]{amsart} 
\usepackage{palatino, epic,eepic, amssymb, xypic, floatflt, microtype, hyperref}
\usepackage{verbatim,color}
\usepackage{enumitem}

\usepackage{tikz}
\usetikzlibrary{decorations.markings}
\definecolor{Burgundy}{RGB}{144,0,32}
\usepackage{float}
\input xy
\xyoption{all}

\newtheorem{tm}{Theorem}[section]
\newtheorem{pr}[tm]{Proposition}
\newtheorem{lm}[tm]{Lemma}
\newtheorem{co}[tm]{Corollary}

\newtheorem{rem}[tm]{Remark}

\newcommand{\Q}{{\mathbb Q}}
\newcommand{\C}{{\mathbb C}}

\newcommand{\Z}{{\mathbb Z}}
\newcommand{\z}{{\zeta}}

\newcommand{\e}{\varepsilon}

\usepackage{colonequals}
\usetikzlibrary{calc}
\usetikzlibrary{decorations.pathreplacing}
\usetikzlibrary{arrows,chains,matrix,positioning,scopes}

\makeatletter
\newcommand{\oset}[3][0ex]{%
  \mathrel{\mathop{#3}\limits^{
    \vbox to#1{\kern-2\ex@
    \hbox{$\scriptstyle#2$}\vss}}}}
\makeatother

\begin{document}

\pagestyle{plain}
\title{Moments of Higher Logarithmic Derivatives and Exceptional Zeros in Cyclotomic Fields}
\author{Samprit Ghosh}
\address{Mathematical Sciences 468, University of Calgary, \newline 2500 University Drive NW, Calgary, Alberta, T2N 1N4, Canada. }
\email{samprit.ghosh@ucalgary.ca}
\thanks{}
\begin{abstract}
Let $\chi$ be a non-principal Dirichlet character, and let $L(s,\chi)$ be the associated Dirichlet $L$-function. We write $\mathcal{L}(s,\chi)$ for its logarithmic derivative $L'(s,\chi)/L(s,\chi)$. In this article, we prove arithmetic formulas for the higher derivatives $\mathcal{L}^{(r)}(1,\chi)$ and establish unconditional moment asymptotics for $P^{(a,b)}(\mathcal{L}^{(r)}(1,\chi))$ as $\chi$ runs over non-principal Dirichlet characters of large prime conductor. As an application, we show that these moment asymptotics imply positivity of the second Li coefficient of prime cyclotomic fields of sufficiently large conductor. Combined with a zero-free criterion in terms of this Li coefficient, this rules out the possible exceptional zero in a Stark-type zero-free region for these fields.
\end{abstract}
\maketitle
\tableofcontents

\section{Introduction}

Let $m$ be a prime number, and let $X_m$ denote the set of all non-principal
primitive Dirichlet characters 
$\chi : (\Z/m\Z)^{\times} \longrightarrow \C^{\times}$.
For $\chi \in X_m$, let $L(s,\chi)$ be the associated Dirichlet $L$-function, and write
\[
\mathcal L(s,\chi) := \frac{L'(s,\chi)}{L(s,\chi)}.
\]
For a pair of non-negative integers $(a,b)$, set $
P^{(a,b)}(z)=z^a\overline{z}^{\,b}$.

The starting point for this paper is a theorem of Ihara, Murty and Shimura
\cite{ihara2}, who studied moments of the logarithmic derivative
$\mathcal L(1,\chi)$. They proved that, for sufficiently large $m$ and any $\epsilon>0$,
\begin{equation}\label{murty318}
	\frac{1}{|X_m|} \sum_{\chi \in X_m} P^{(a,b)}(\mathcal L(1,\chi))
	= (-1)^{a+b}\mu^{(a,b)} + O(m^{\epsilon-1}),
\end{equation}
where
\[
\mu^{(a,b)}
=
\sum_{n=1}^{\infty}
\frac{\Lambda_a(n)\Lambda_b(n)}{n^2},
\qquad
\Lambda_k(n)=\sum_{n=n_1\cdots n_k}\Lambda(n_1)\cdots \Lambda(n_k).
\]
Here $\Lambda$ is the von Mangoldt function. In particular, 
\begin{equation*}
	\lim_{m \rightarrow \infty}	\frac{1}{|X_m|} \sum_{\chi \in X_m} P^{(a,b)} (\mathcal{L}(1, \chi)) = (-1)^{a+b} \mu^{(a,b)}  .
\end{equation*}

The first purpose of this article is to prove a higher-derivative analogue of
this theorem. More precisely, for every integer $r\geq 0$, we study moments of
\[
\mathcal L^{(r)}(1,\chi)
=
\left(\frac{d}{ds}\right)^r \frac{L'(s,\chi)}{L(s,\chi)}\Bigg|_{s=1}
\]
as $\chi$ ranges over $X_m$. We first obtain an arithmetic formula for
$\mathcal L^{(r)}(1,\chi)$, extending the corresponding formula for
$\mathcal L(1,\chi)$ in \cite{ihara2}. We then prove asymptotic formulas for
the averages of $P^{(a,b)}(\mathcal L^{(r)}(1,\chi))$, first under GRH and then
unconditionally. Moments of higher derivatives of $L(s,\chi)$ have been studied extensively; see, for example, work of Conrey
\cite{conrey}, Milinovich \cite{milin}, Heath-Brown \cite{brown},
Soundararajan \cite{sound1}, and Sono \cite{sono}. Thus similar results for
higher derivatives of the logarithmic derivative at $s=1$ might be of independent interest. However, our main motivation for this study comes from a different perspective. 

The second purpose of the paper is to connect these moment asymptotics with Li's
coefficients and zeros of zeta and $L$-functions. For a number field $K$, Li in \cite{licoeff} introduced the sequence of numbers  
\begin{equation}
	\lambda_n (K) = \frac{1}{(n-1)!} \frac{d^n}{ds^n} \left(s^{n-1} \log \xi_K(s) \right) \Big|_{s=1} \text{ \hspace{1cm} for } n \geq 1, 
\end{equation}
where $\xi_K(s)$ is the completed Dedekind zeta function of $K$. He proved the following result. \\
\textbf{Theorem (Li's Criterion).} 
The generalized Riemann hypothesis for $\z_K(s)$ holds if and only if $\; \lambda_n (K) \geq 0$  for all $n \geq 1$.\\ Later, Li's original criterion was extended to automorphic $L$-functions by Lagarias in \cite{lagariasli}. \vspace{2mm}\\
\textbf{Motivating Question.} The signs of individual Li coefficients are delicate and difficult to control. Is their average behavior more accessible? More precisely, for a fixed modulus $q$, what can be said about the Li coefficients averaged over Dirichlet characters $\chi$ mod $q$?\\

The moment problem studied here can be viewed as a first step toward this average theory of Li-type coefficients. Finally, we give an application of our asymptotics in the cyclotomic setting.

For the prime cyclotomic field $K_m=\Q(\zeta_m)$, the factorization of
$\zeta_{K_m}(s)$ into Dirichlet $L$-functions expresses the second Li
coefficient $\lambda_2(K_m)$ in terms of averages of $\mathcal L(1,\chi)$ and
$\mathcal L'(1,\chi)$. The unconditional moment theorem proved here shows that
these averages are negligible compared with the main $\log m$ term. As a
consequence, $\lambda_2(K_m)>0$ for all sufficiently large primes $m$.
Combining this positivity statement with a zero-free criterion rules out the possible exceptional zero in a Stark-type zero-free region.

We now state the main results.
\subsection{Statements of Main results}
Let $K$ be a number field and $\chi$ be a primitive Dirichlet character over $K$ (that is, a primitive Hecke character of finite order). Let $L(s, \chi)$ be the $L$-function associated with it. In particular, when $\chi = \chi_0$, the principal character, $L(s, \chi) = \zeta_K(s)$, the Dedekind zeta function of $K$.
\begin{tm}\label{mainthm1}
	For $\chi \neq \chi_0$ and $r\geq 0$, we have, unconditionally,
	$$\mathcal{L}^{(r)}(1, \chi) =  \lim_{x \rightarrow \infty} (-1)^{r+1} \;  \Phi_K (\chi, r , x),$$
	where $$ \Phi_K(\chi, r,x) = \frac{1}{(x-1)}\sum \limits_{k, \; N(P)^k< \; x}  \left( \dfrac{x}{N(P)^k} -1 \right)  k^r\chi(P)^k (\log N(P))^{r+1}.$$
\end{tm}

For results on moments, we focus on $K=\Q$ and introduce some more notation. We define  
\begin{equation}\label{defnLambda}
	\Lambda_{r, k}(n) = \sum_{n = n_1 \cdots n_k} \Lambda(n_1) \cdots \Lambda(n_k) (\log n_1)^r \cdots (\log n_k)^r \;\;\;\; \text{ for } k>0.
\end{equation}
For $k = 0$, define $\Lambda_{r, 0}(n) = 1$ if $n=1$ (irrespective of $r$) and $0$ otherwise. Also, let $m$ be a large prime number, and let $X_m$ be the collection of all non-principal primitive Dirichlet characters $\chi : (\Z / m\Z)^{\times} \rightarrow \C^{\times}$.
\begin{tm}\label{mainthm2}
	Under GRH, for every $r \geq 0$ and every pair of non-negative integers $(a,b)$, we have
	\begin{equation*}
		\frac{1}{|X_m|} \sum_{\chi \in X_m} P^{(a,b)} (\mathcal{L}^{(r)}(1, \chi)) = (-1)^{(r+1)(a+b)} \mu^{(a,b)}(r) + O \left( \frac{(\log m)^{(r+1)(a+b)+2}}{m} \right).
	\end{equation*}
	Here, the implicit constant depends on $a, b$. In particular, 
	\begin{equation}\label{limitformula}
		\lim_{m \rightarrow \infty}  	\frac{1}{|X_m|} \sum_{\chi \in X_m} P^{(a,b)} (\mathcal{L}^{(r)}(1, \chi)) \; = \;  (-1)^{(r+1)(a+b)} \mu^{(a,b)}(r). \;\;\;\;\;\;\;\;\;\;\;\;\;\;\;\;
	\end{equation}
	Here $\mu^{(a,b)}(r)$ is defined as  
	$$ \mu^{(a,b)}(r) = \sum_{j=1}^{\infty} \frac{ \Lambda_{r,a}(j) \;  \Lambda_{r,b}(j)}{j^2}.$$
\end{tm}
\begin{tm}\label{mainthm3}
	Unconditionally, for every $r \geq 0$ and every pair of non-negative integers $(a,b)$, we have 
	\begin{equation*}
		\frac{1}{|X_m|} \sum_{\chi \in X_m} P^{(a,b)} (\mathcal{L}^{(r)}(1, \chi)) = (-1)^{(r+1)(a+b)} \mu^{(a,b)}(r) + O_{r,a,b} \left( m^{\e - 1}\right),
	\end{equation*}
	for any $\e > 0$. In particular, the limit formula in \eqref{limitformula} is unconditional.
\end{tm}
A well-known result of Stark \cite[Lemma 3]{stark1} says that, for $K \neq \Q$, $\z_K(s)$ has at most one zero in the region
\begin{equation*}
	1 - \dfrac{1}{4 \log d_K} \leq \Re(s) \leq 1, \;\;\;\; |\Im(s)| \;  \leq \dfrac{1}{4 \log d_K}.
\end{equation*}

Moreover, if this zero exists, it is necessarily real and simple. The next result is obtained by applying Theorem \ref{mainthm3} with
$r=0$ and $r=1$ to the factorization of the Dedekind zeta function of a
cyclotomic field. The resulting positivity of $\lambda_2(K_m)$ rules out the
possible exceptional zero in a wider Stark-type region, as shown later in Section \ref{sec4}, Theorem \ref{mainthm4}.
\begin{tm}\label{mainthm5}
	Let $m$ be a prime number and $K_m = \Q(\zeta_m)$ denote the cyclotomic extension. For $m$ sufficiently large, the Dedekind zeta function $\z_{K_m}(s)$ has no zeros in the region 
	\[
	1 -\dfrac{1}{3 \log d_{K_m}} < \text{Re}(s) < 1, \;\; | \text{Im}(s)| < \dfrac{1}{3 \log d_{K_m}},
	\]
	where $d_{K_m} = m^{m-2}$ is the absolute value of the discriminant. In particular, the exceptional zero does not exist for sufficiently large $m$.
\end{tm}
\begin{co}\label{maincor1}
	For a sufficiently large prime number $m$, consider the quadratic fields  
	\[
	\Q_m =
	\begin{cases}
		\Q(\sqrt m), & \text{ if }m\equiv 1\pmod 4,\\
		\Q(\sqrt{-m}), & \text{ if }m\equiv 3\pmod 4.
	\end{cases}
	\]
	The Dedekind zeta function $\z_{Q_m}(s)$ has no real zeros in the region 
	\[
	1 -\dfrac{1}{3(m-2) \log m} < \text{Re}(s) < 1.
	\]
	In particular, the exceptional zero $\beta$, if it exists, satisfies $\beta < 1 -\dfrac{1}{3(m-2) \log m}$.
\end{co}

\section*{Acknowledgments.}
The author sincerely thanks Prof. V. Kumar Murty for his valuable suggestions, encouragement and many insightful discussions during the developement of this work. The author is currently partially supported by NSERC grants RGPIN-2018-03770, RGPIN-2020-06075, and by CRC tier-2 research stipend 950-231716 at the University of Calgary. The author thanks Prof. Khoa D. Nguyen for their financial support and mentorship.

\section{An explicit arithmetic formula for $\mathcal{L}^{(r)}(1, \chi)$}\label{Momexact}
Let $K$ be a number field and $\chi$ be a primitive Dirichlet character over $K$ (i.e. a primitive Hecke character of finite order). Let $L(s, \chi)$ be the $L$-function associated with it. In particular, when $\chi = \chi_0$, the principal character, $L(s, \chi) = \zeta_K(s)$, the Dedekind zeta function of $K$. The completed $L$-function has an expression of the form
\begin{equation} \label{funceqn}
\xi(s,\chi) = A B^{\frac{s}{2}} \Gamma \left(\dfrac{s+1}{2}\right)^a \Gamma \left(\dfrac{s}{2}\right)^{a'}\Gamma(s)^{r_2} \; L(s, \chi),
\end{equation}
and satisfies a functional equation $\xi(s,\chi) = W(\chi) \xi(1-s, \overline{\chi})$, where $W(\chi)$ has absolute value $1$. Here $A$ and $B$ are constants involving $2$, $\pi$, the discriminant $d_K$ of $K$, and the conductor  $\mathfrak{f}_{\chi}$. We are not concerned with writing $A$ and $B$ explicitly here, since in this article we are only concerned with higher derivatives of $\xi'(s,\chi)/\xi(s,\chi)$. Interested readers are directed to p. 211 of \cite{cassels} or to Hecke's original proof \cite{hecke} for details concerning $A$ and $B$.

Here $a$ (respectively $a'$) is the number of real places of $K$ where $\chi$ is ramified (resp. unramified), $r_1 = a+a'$ is the number of real places of $K$ and $r_2$ is the number of complex places. For $\chi \neq \chi_0$, taking the logarithmic derivative of (\ref{funceqn}) and using the Hadamard product one can deduce a Stark-like lemma (see Lemma 2.1 of \cite{stark1} or p. 83 of \cite{dave})
  \begin{equation} \label{starklike}
\dfrac{L'(s, \chi)}{L(s, \chi)} = C - \frac{a}{2} \frac{\Gamma'}{\Gamma} \left( \dfrac{s}{2}\right)  -  \frac{a'}{2} \frac{\Gamma'}{\Gamma} \left( \dfrac{s+1}{2}\right) - r_2  \frac{\Gamma'}{\Gamma} \left( s \right) + \sum_{\rho} \left( \dfrac{1}{s - \rho} + \dfrac{1}{\rho} \right).
  \end{equation}
Here $C$ is a constant involving $\log$ of terms in $B$ in \eqref{funceqn}. For brevity, we write $$\mathcal{L}(s, \chi) = 	\dfrac{L'(s, \chi)}{L(s, \chi)}  \;\;\;\;\;\; \text{ (say)}.$$
On the other hand, we can take the logarithmic derivative of the Euler product of  $L(s,\chi)$ to get
\begin{equation}\label{eulerprod}
	\mathcal{L}(s, \chi) = - \sum_{P, k} \left(\dfrac{\chi(P)}{N(P)^s} \right)^k \log N(P).
\end{equation}
To obtain the arithmetic formula, we follow computations very similar to those in the proof of Theorem 1.8 in \cite{ghosh1}.

The following arithmetic formulas and bounds were proved in \cite[Corollary 2.2.5]{ihara2}. \\
\textbf{Theorem.} \emph{(Ihara, Murty, Shimura)} \\
\emph{If $\chi \neq \chi_0$, then 
	\begin{equation}
		\mathcal{L}(1,\chi)  = - \; \lim_{x \rightarrow \infty} \Phi_{K, \chi}(x),
	\end{equation}
	where 
	\begin{equation*}
		\Phi_{K, \chi} (x) = \dfrac{1}{x - 1} \sum_{N(P)^k \leq x} \left(\dfrac{x}{N(P)^k} - 1 \right) \chi(P)^k \; \log N(P) \;\;\;\; ( \text{ for } x>1).
	\end{equation*}
	Here, $k$ is a positive integer and the sum is taken over non-Archimedean primes. } \\

Under GRH, they have also shown the upper bound (see \cite[Theorem 3]{ihara2})  
\begin{equation*}
	\left| \mathcal{L}(1,\chi)  \right| < 2 \; \log \log \sqrt{d_{\chi}} \; + 1 - \gamma_{K} + \; O \left( \frac{ \log |d_K| + \log \log d_{\chi}}{\log d_{\chi}}\right).
\end{equation*}
Here, $\chi$ is a non-principal primitive Dirichlet character of a number field $K$ with conductor $\mathfrak{f}_{\chi}$. Furthermore, $d_{\chi} = |d_K| N(\mathfrak{f}_{\chi})$ and $\gamma_{K}$ is the Euler-Kronecker constant of $K$. \\

For $r \geq 1$, we take the $r$-th derivative of (\ref{eulerprod}) to get
\begin{equation}\label{eulerprod4}
	\mathcal{L}^{(r)}(s, \chi) =  (-1)^{r+1} \sum_{P, k} k^r \left(\dfrac{\chi(P)}{N(P)^s} \right)^k (\log N(P))^{r+1}.
\end{equation} 
Similarly differentiating (\ref{starklike}) $r$ times, 
\begin{equation}\label{starkliked}
	\mathcal{L}^{(r)}(s, \chi) = (-1)^r r! \sum_{\rho} \frac{1}{(s-\rho)^{r+1} } \; + \; \tilde{\Gamma}_{\chi}^{(r)}(s).
\end{equation}
Thus, letting $s\rightarrow 1$, we get 
\begin{equation}\label{starklike2}
    \mathcal{L}^{(r)}(1, \chi) = (-1)^r r! \sum_{\rho} \frac{1}{(1-\rho)^{r+1} } \; + \; \tilde{\Gamma}_{\chi}^{(r)}(1).
\end{equation}
In a previous paper \cite{ghosh1}, the author introduced the higher Euler-Kronecker constants. These are the coefficients $\gamma_{K,r}$ that appear in the Laurent series expansion of the logarithmic derivative of the Dedekind zeta function about $s=1$ : 
$$ \dfrac{\z_K'(s)}{\z_K(s)} = \dfrac{-1}{s-1}\; + \; \gamma_{K,0} \; + \; \sum_{r=1}^{\infty} \gamma_{K,r}(s-1)^r \;.$$ To prove Theorem \ref{mainthm1}, we will closely follow its counterpart for the Dedekind zeta function, as in \cite[Section 5]{ghosh1}. 
We will evaluate the integral
$$\Psi_{\chi}(\mu, r, x) = \dfrac{1}{2\pi i}\int_{c-i\infty}^{c+i \infty} \dfrac{x^{s-\mu}}{s-\mu} \mathcal{L}^{(r)}(\chi, s) \;ds \;\;\;\;\; \text{ for } c\gg 0,$$ in two different ways using equations (\ref{eulerprod4}) and  (\ref{starkliked}), for $\mu = 0$ and $1$. Using \eqref{eulerprod4} we get
 \begin{align}\label{eulerprod5}
\nonumber	& x \Psi_{\chi}(1, r, x) - \Psi_{\chi}(0, r, x) \\
& \;\;\; =    (-1)^{r+1} \sum \limits_{k, \; N(P)^k< \; x}  \left( \dfrac{x}{N(P)^k} -1 \right)  k^r\chi(P)^k (\log N(P))^{r+1}.  \;\;\;
\end{align}
The details for the above are the same as those in \cite[Lemma 5.1]{ghosh1}, except for an extra $\chi(P)^k$ factor, and a negative sign. This motivates the following definition
\begin{equation}\label{defnPhichi}
    \Phi_K(\chi, r,x) = \frac{1}{(x-1)}\sum \limits_{k, \; N(P)^k< \; x}  \left( \dfrac{x}{N(P)^k} -1 \right)  k^r\chi(P)^k (\log N(P))^{r+1}.
\end{equation}

On the other hand, we can similarly compute the contribution from the $\sum_{\rho}$ term and $\Gamma$-factor. Let us denote 
\begin{equation}\label{Contzero}
 \mathfrak{Z}_\chi(r,x)= \frac{(-1)^{r} r!}{2 \pi i} \int_{c - i \infty}^{c + \infty} \sum x^s \left[ \frac{1}{(s-1)(s- \rho)^{r+1}} - \frac{1}{s(s-\rho)^{r+1}}\right]ds, \text{ and}
 \end{equation}
 \begin{equation} \label{Contgamma}
 	\mathfrak{G}_\chi(r,x) = \frac{x}{2 \pi i} \int_{c - i \infty}^{c + \infty} \frac{x^{s-1}}{s-1 } \tilde{ \Gamma}^{(r)}_{\chi}(s) \; ds \; - \;  \frac{1}{2 \pi i} \int_{c - i \infty}^{c + \infty} \frac{x^{s}}{s } \tilde{ \Gamma}^{(r)}_{\chi}(s) \; ds.
 \end{equation}

For the nontrivial zeros, we carry out contour computations similar to those presented in Lemma 5.5 of \cite{ghosh1} and the discussion before the lemma to deduce
\begin{align} \label{Rhofac}
	\nonumber	\frac{\mathfrak{Z}_{\chi}(r,x)}{x-1} & = r!  (-1)^{r}  \sum_{\rho} \frac{1}{(1- \rho)^{r+1}} + \frac{[(-1)^{r} - 1]}{(x-1)}\sum \dfrac{r!}{\rho^{r+1}} \\ \nonumber &+\sum_{\rho} \left( (-1)^{r} \sum_{k=0}^{r} (-1)^k k! \binom{r}{k} \left[ \frac{1}{(\rho-1)^{k+1}} - \frac{1}{\rho^{k+1}} \right] \frac{ x^{\rho} (\log x)^{r-k}}{(x-1)}\right). \\
        &= r!  (-1)^{r}  \sum_{\rho} \frac{1}{(1- \rho)^{r+1}} \; + \; E_1(r,x) \;\; \text{(say)}.
	\end{align}
    
Furthermore, we can make use of the standard zero-free regions of Hecke L-functions (see \cite[Lemma 2.3]{LagOd}) to ensure that the rest of the argument can be carried out similarly, i.e. $\lim_{x \rightarrow \infty} E_1(r,x) = 0$. Thus, for $r \geq 1$, we have
	$$\lim_{x \rightarrow \infty}\frac{\mathfrak{Z}_{\chi}(r,x)}{(x-1)} = (-1)^{r} r!  \sum \frac{1}{(1- \rho)^{r+1}}.   $$
We also note that, under GRH 
\begin{equation}
   E_1(r,x) = O \left( \frac{(\log x)^r}{\sqrt{x}}(r! \; 2^{r-1}\log d_\chi)  \right),
\end{equation}
where $d_{\chi} = |d_K|N(\mathfrak{f}_{\chi})$. This estimate can be easily seen from \cite[Theorem 2]{ihara2}.

To compute the contribution from the Gamma terms, we first note that 
\begin{align*}
	\tilde{\Gamma}_{\chi}(s) & =  -  \frac{a}{2} \frac{\Gamma'}{\Gamma} \left( \dfrac{s+1}{2}\right)  - \frac{a'}{2} \frac{\Gamma'}{\Gamma} \left( \dfrac{s}{2}\right) - r_2  \frac{\Gamma'}{\Gamma} \left( s \right) \\
	&=  \frac{n_K}{2} \gamma  \; + \frac{a}{2} \sum_{k=0}^{\infty} \;  \left( \frac{2}{s+1+2k} -   \frac{2}{2k+2} \right)\;  + \frac{a'}{2}  \sum_{k=0}^{\infty} \; \left(  \frac{2}{s+2k} - \frac{2}{1+2k}\right) \; \\
	& \;\;\;\;\;\;\; \;\;\;\; +  \; r_2 \sum_{k=0}^{\infty} \; \left( \frac{1}{s+k} - \frac{1}{1+k} \right).
\end{align*}
Here $\gamma = \lim_{n \rightarrow \infty } \left[ \sum_{k=1}^{n}\frac{1}{k} - \ln n \right] \sim 0.5772\ldots $ is the Euler constant, and $n_K = [K:\Q]$. 
Differentiating $r$ times, we get 
\begin{align*}
\tilde{\Gamma}_{\chi}^{(r)}(s) = (-1)^{r} r! \left[ \frac{(a' + r_2)}{s^{r+1}} \right.  & + a \sum_{k = 0}^{\infty}  \frac{1}{(s+1+2k)^{r+1}} \\
& \left. + a'  \sum_{k = 1}^{\infty}  \frac{1}{(s+2k)^{r+1}} + r_2  \sum_{k = 1}^{\infty}  \frac{1}{(s+k)^{r+1}} \right]. \\
\end{align*} 
Thus, by performing a computation similar to that in \cite[Lemma 5.6]{ghosh1} we obtain
\begin{align}\label{Gammafac}
		\frac{\mathfrak{G}_{\chi}(r,x)}{(x-1)} &= \tilde{\Gamma}_{\chi}^{(r)}(1) + E_2(r,x),
	\end{align}
 where $E_2(r,x) = O_{r} \left( \frac{n_K (\log x)^{r+1}}{x} \right)$, and so $\lim_{x \rightarrow \infty} E_2(r,x) = 0$. In particular,
    \begin{equation}
        \lim_{x \rightarrow \infty} \frac{\mathfrak{G}_{\chi}(r,x)}{(x-1)} = \tilde{\Gamma}_{\chi}^{(r)}(1).
    \end{equation}

 \subsection{Proof of Theorem \ref{mainthm1}} $\;$ \\    
 Putting \eqref{defnPhichi}, \eqref{Rhofac} and \eqref{Gammafac} together, we get
 \begin{equation}\label{mainarithformula}
     (-1)^{r+1}\Phi_K (\chi, r , x) = (-1)^{r} r!  \sum \frac{1}{(1- \rho)^{r+1}} + \tilde{\Gamma}_{\chi}^{(r)}(1) + E_1(r,x) + E_2(r,x). 
 \end{equation}
 Letting $x \rightarrow \infty$ and using \eqref{starklike2}, we get the arithmetic limit formula 
$$\mathcal{L}^{(r)}(1, \chi) =  \lim_{x \rightarrow \infty} (-1)^{r+1} \;  \Phi_K (\chi, r , x).$$
$\;$ \hfill $\qed$
\begin{rem}
	Note the difference in the limit formula compared with (1.8) of \cite[Theorem 1.4]{ghosh1}; in particular, the absence of the $f(r,x)$ term.
\end{rem}

We will now focus our attention on the case $K = \Q$. In the following sections, we study the moments of higher derivatives of $\mathcal{L}(s, \chi)$ at $s=1$, where $\chi$ runs over all non-principal Dirichlet characters of large prime conductors. 
 
\section{Moments of higher derivatives $\mathcal{L}^{(r)}(1, \chi)$}

Recall that we wrote  
\begin{equation*}
	 \Phi_K (\chi, r , x) =  \frac{1}{x-1} \sum \limits_{k, \; N(P)^k< \; x} k^r \left( \dfrac{x}{N(P)^k} -1 \right)  \chi(P)^k (\log N(P))^{r+1}.  
\end{equation*}
In particular, for $K= \Q$, we have 
\begin{align*}
	\Phi_{\Q} (\chi, r , x) 
	 &=  \frac{1}{x-1} \sum \limits_{k, \; p^k< \; x} k^r \left( \dfrac{x}{p^k} -1 \right)  \chi(p)^k (\log p)^{r+1}  \\[1.3ex]
	 &= \frac{1}{x-1} \sum \limits_{k, \; p^k< \; x}  \left( \dfrac{x}{p^k} -1 \right)  \chi(p^k) (\log p)(\log p^k)^{r}  \\[1.3ex]
	  &= \frac{1}{x-1} \sum \limits_{n < \; x}  \left( \dfrac{x}{n} -1 \right)  \chi(n) \Lambda(n)(\log n)^{r}.  \\
\end{align*}
From now on, we will drop the subscript $\Q$ and simply write $\Phi (\chi, r , x)$ to denote the case $K=\Q$. That is, we will write 
\begin{equation}\label{Phi_r_defn}
    \Phi (\chi, r , x)  = \frac{1}{x-1} \sum \limits_{n < \; x}  \left( \dfrac{x}{n} -1 \right)  \chi(n) \Lambda(n)(\log n)^{r} .
\end{equation}
We will first compute the $(a,b)$-th moment of $\Phi (\chi, r , x)$. Before that, we observe some bounds on the $\Lambda_{r,k}$ terms that we defined in \eqref{defnLambda}.

Recall $\Lambda_{r,k}$  is the higher derivative counterpart to $\Lambda_{k}(n)$ which was defined as 
$$ \Lambda_0(n) = 1 \text{ if } n=1, \; 0 \text{ otherwise. } \quad \Lambda_k(n) = \sum_{n = n_1 \cdots n_k} \Lambda(n_1) \cdots \Lambda(n_k) \; \text{ for } k>0.$$
We see that if $n$ has the prime factorization $n = \prod_{i=1}^r p_i^{\alpha_i}$ then $\Lambda_{k}(n)$ is the coefficient of the monomial $x_1^{\alpha_1} \cdots x_r^{\alpha_r}$ in the polynomial 
$$\left( \sum_{i=1}^{r} (\log p_i)(x_i + x_i^2 + \cdots + x_i^{\alpha_i}) \right)^k.$$
Letting $x_i = 1$ for all $i = 1 , \cdots r$, we see that
\begin{equation}
	 \Lambda_{k}(n) \leq \left( \sum_{i=1}^{r} {\alpha_i} (\log p_i)  \right)^k = (\log n)^k. 
\end{equation}
Applying the arithmetic-geometric mean inequality, we see that 
$$\prod_{i=1}^{k}\log n_i \leq \frac{(\log n )^k}{k^k} .$$
Therefore, we have
\begin{align}\label{lambdaineq1}
	\nonumber \Lambda_{r,k}(n)&=	 \sum_{n = n_1 \cdots n_k} \Lambda(n_1) \cdots \Lambda(n_k) (\log n_1)^r \cdots (\log n_k)^r\\ 
	 & \leq \frac{(\log n )^{rk}}{k^{rk}} \Lambda_{k}(n) \leq \frac{(\log n )^{(r+1)k}}{k^{rk}}.
\end{align}

We now prove the following lemma, which is a generalization of 4.2.2 and 4.2.3 of \cite{ihara2}. 
\begin{lm}\label{genOrtho}
	For $x > 1$, let $g(x,n)$ be any real-valued function and let $g_\chi(x) = \sum_{n \leq x} g(x, n) \chi(n) $. For $k \geq 1$, define   $$\lambda^{(k)}(j, x) =  \; \sum \limits_{\substack{n_1, \cdots, n_k < x \\ n_1 \cdots n_k \equiv j \; (\text{mod} \; m) }} \prod_{i=1}^k g(x, n_i), $$
 and for $k=0$ define $\lambda^{(0)} (j, x) = 1$ for $j=1$, and $0$ for $j>1$. Then we have
\begin{equation} \label{orthorel}
		\frac{1}{|X_m^{\star}|} \sum_{\chi \in X_m^{\star}} g_{\chi}(x)^a g_{\overline{\chi}}(x)^b \; = \sum_{j=1}^{m-1} \lambda^{(a)}(j,x) \lambda^{(b)}(j,x).
\end{equation}
  (Recall that $m$ here is a prime number, $a$ and $b$ are non-negative integers, and $X_m^{\star} = X_m \setminus \{ \chi_0\}$.)    
\end{lm}
\begin{proof}
	This is a direct consequence of the orthogonality relations of Dirichlet characters. In particular, a typical term of the sum on the left-hand side of (\ref{orthorel}) looks like $$\left( \prod_{i=1}^{a} g(x, n_i)  \prod_{j=1}^{b}g(x, m_j) \right) \chi(n_1 \cdots n_a) \overline{\chi}(m_1 \cdots m_b).$$ When this is summed over all $\chi$, it has a nonzero contribution only when $(n_1 \cdots n_a) \equiv (m_1 \cdots m_b) $ (modulo $m$) and hence we have our result.
\end{proof}
\begin{pr}\label{momchigen}
For each pair $(a,b)$ of non-negative integers, each $r \geq 0$, and each $x \geq m$, we have
\begin{equation}
	\frac{1}{|X_m|} \; \sum_{\chi \in X_m}P^{(a,b)} \; \left( \Phi( \chi, r, x) \right) = \; {\mu}^{(a,b)}(r) + O_{a,b} \left( \frac{(\log x)^{(r+1)(a+b) + 2}}{m}\right).
\end{equation}
\end{pr}
\begin{proof}
	 Applying Lemma \ref{genOrtho} with $g(x,n) = \frac{1}{(x-1)} \left( \frac{x}{n} - 1 \right) \Lambda(n) (\log n)^r $ we get, 
	\begin{equation}
			\frac{1}{|X_m^{\star}|} \sum_{\chi \in X_m^{\star} } \; P^{(a,b)} ( \Phi(\chi, r, x)) = \sum_{j=1}^{m-1} \lambda^{(a)}(j,x) \; \lambda^{(b)}(j,x),
	\end{equation}
where \begin{align}
\nonumber	\lambda^{(k)}(j,x) \; &= \sum_{\substack{n_1, \cdots , n_k < x \\ n_1 \cdots n_k \equiv j \; (\text{mod } m)}}  \prod_{i=1}^k \frac{1}{(x-1)} \left( \frac{x}{n_i} - 1 \right) \Lambda(n_i) (\log n_i)^r \\[1.3ex]
&= \frac{1}{(x-1)^k} \sum_{l=0}^{[(x^k - j)/m]}   \sum_{\substack{n_1, \cdots , n_k < x \\ n_1 \cdots n_k = j+l m}}  \prod_{i=1}^k  \left( \frac{x}{n_i} - 1 \right) \Lambda(n_i) (\log n_i)^r. 
\end{align}
Let us write 
\begin{equation}\label{sumL_k}
    \lambda^{(k)}(j,x) = \sum_{l=0}^{[(x^k - j)/m]} L^{(k)} (j + lm, x).
\end{equation}
We will show that the net contribution coming from all $l>0$ terms together is small. For this, we note that $L^{(k)}(N,x) \neq 0$ only when $N<x^k$ and in this case, 
\begin{align*}
L^{(k)}(N, x) & \leq \frac{1}{N} \sum_{\substack{n_1, \cdots , n_k < x \\ n_1 \cdots n_k =N}}  \left( \prod_{i=1}^{k} \Lambda(n_i)(\log n_i)^r \right) 
 \leq  \frac{1}{N} \Lambda_{r,k}(N)  \leq k^k \frac{(\log x)^{(r+1)k}}{N}.
\end{align*}
The last inequality follows from \eqref{lambdaineq1} and $N<x^k$.
Therefore, we have 
\begin{align*}
\sum_{l=1}^{[(x^k - j)/m]} L^{(k)} (j + lm, x) &< k^k \frac{(\log x)^{(r+1)k}}{m} \left( 1 + \frac{1}{2} + \cdots + \frac{1}{[x^k/m]}\right)\\[1.3ex]
&= \; O\left(\frac{(\log x)^{(r+1)k+1}}{m} \right).
\end{align*}
For the main term, we first prove 
\begin{equation}\label{xbyni}
\frac{1}{(x-1)^k} \prod_{i=1}^{k} \left( \frac{x}{n_i} - 1 \right) = \frac{1}{j} + O \left( \frac{1}{x} \right), \;\; \text{here } j=n_1 \cdots n_k. 
\end{equation}
We first note that for $x>0$ and for any integers $u, v \geq 1$, we have $$(x-u)(x-v) \geq (x-1)(x-uv).$$ 
Generalizing, 
$$ (x-1)^k \geq (x-n_1) \cdots (x-n_k) \geq (x-1)^{k-1}(x-n_1 \cdots n_k).$$
Thus, for $n_i \geq 1$ and $n_1 \cdots n_k = j$, 
\begin{align*}
	\frac{1}{(x-1)^k} \prod_{i=1}^{k} \left( \frac{x}{n_i} - 1 \right) & \geq \frac{ \prod_{i=1}^{k} (x - n_i)}{(x-1)^k \; j }    \geq \frac{1}{(x-1)} \frac{x-j}{j}, \;\;\; \text{ and so, }\\[1.3ex]
 \; \frac{1}{j} - \frac{1}{(x-1)^k} \prod_{i=1}^{k} \left( \frac{x}{n_i} - 1 \right) &\leq \frac{j-1}{j(x-1)}.
\end{align*}
On the other hand,
\begin{align*}
	\frac{1}{(x-1)^k} \prod_{i=1}^{k} \left( \frac{x}{n_i} - 1 \right) &= 	\frac{1}{(x-1)^k} \frac{ \prod_{i=1}^{k} (x-n_i)}{j} \leq \frac{1}{j}.
\end{align*}
Therefore,
\begin{equation}
\frac{1}{(x-1)^k} \prod_{i=1}^{k} \left( \frac{x}{n_i} - 1 \right) = \frac{1}{j} + O \left( \frac{1}{x} \right).
\end{equation}
For the $l=0$ term of \eqref{sumL_k}, we have
\begin{equation*}
	L^{(k)}(j,x) = \frac{1}{(x-1)^k}\sum_{\substack{n_1, \cdots , n_k < x \\ n_1 \cdots n_k = j}}  \prod_{i=1}^k  \left( \frac{x}{n_i} - 1 \right) \Lambda(n_i) (\log n_i)^r . 
\end{equation*}
Since $j<m$, if we choose $x \geq m$, the condition $n_1, \cdots, n_k < x$ is automatic. Thus, using \eqref{xbyni} we have
\begin{equation}
 L^{(k)}(j,x) = \frac{ \Lambda_{r,k}(j)}{j} + O \left( \frac{(\log m)^{(r+1)k}}{m} \right).  
\end{equation}
Hence 
$$	\lambda^{(k)}(j,x) = \frac{\Lambda_{r,k}(j)}{j} + O \left( \frac{(\log x)^{(r+1)k+1}}{m} \right),$$
and so,
\begin{align*}
\frac{1}{|X_m^{\star}|} \sum_{\chi \in X_m^{\star} } \; P^{(a,b)} ( \Phi(\chi, r, x)) &=     \sum_{j=1}^{m-1} \lambda^{(a)}(j,x) \; \lambda^{(b)}(j,x) \\
&= \sum_{j=1}^{m-1} \frac{ \Lambda_{r,a}(j) \;  \Lambda_{r,b}(j)}{j^2} + O \left( \frac{(\log x)^{(r+1)(a+b) + 2}}{m^2}\right).
\end{align*}
Finally, we get the proposition from the following two observations. First, 
\begin{align*}
	\sum_{j \geq m} \frac{\Lambda_{r,a}(j) \; \Lambda_{r,b}(j)}{j^2} & \leq \frac{1}{(a^a b^b)^r} \sum_{j \geq m} \frac{(\log j)^{(r+1)(a+b)}}{j^2} = \; O \left( \frac{(\log m)^{(r+1)(a+b) +1}}{m}\right),
\end{align*} 
where the first inequality follows from \eqref{lambdaineq1}.
Second, so far we have excluded $\chi_0$ from our computations and have worked with $X_m^{\star}$. This does not affect the argument, because for $\chi = \chi_0$, $\Phi(\chi_0,r,x) = \Phi_{\Q}(r,x) = O((\log x)^{r+1})$. For an explanation, see \cite[Remark 5.3]{ghosh1}. Thus, if we include the principal character in proving the theorem, it will affect the result by $O \left(\frac{(\log x)^{(r+1)(a+b)}}{m}\right)$, which is smaller than the error term. 
\end{proof} 
We are now ready to prove Theorem \ref{mainthm2}.
\subsection{Proof of Theorem \ref{mainthm2}}
Note that for $K=\Q$, equation \eqref{starklike} takes the form 
\begin{equation}
	\dfrac{L'(s, \chi)}{L(s, \chi)} = - \frac{1}{2} \log \frac{q}{\pi} - \frac{1}{2} \frac{\Gamma'}{\Gamma} \left( \dfrac{s+a}{2}\right) + B(\chi) + \sum_{\rho} \left( \dfrac{1}{s - \rho} + \dfrac{1}{\rho} \right).
\end{equation}
For example, see p. 83 of \cite{dave}. Here $a = 0$ (respectively, $a=1$) if $\chi$ is even (resp. odd) and $B(\chi) = \xi'(0, \chi) / \xi(0, \chi)$. The sum is over all nontrivial zeros $\rho$ of $L(s, \chi)$, i.e., zeros in the critical strip.  \\
Thus, carrying out computations similar to those in the proof of Theorem \ref{mainthm1}, under GRH, we get
\begin{equation}\label{LchiQ}
	\mathcal{L}^{(r)}(1, \chi) = (-1)^{r+1}\Phi(\chi, r, x) + O \left( \frac{\log m \; (\log x)^r}{\sqrt{x}} + \frac{(\log x)^{r+2}}{x}\right). 
\end{equation}
 Substituting $x=m^2$ into \eqref{LchiQ} and applying Proposition \ref{momchigen} completes the proof. We get
 
 \begin{align*}
&	\frac{1}{|X_m|} \sum_{\chi \in X_m} P^{(a,b)}(\mathcal{L}^{(r)}(1,\chi)) \\
& \;\;\; = 	\frac{1}{|X_m|} \sum_{\chi \in X_m} P^{(a,b)} \left( (-1)^{r+1} \Phi(\chi, r, m^2) \right) + O \left( \frac{(\log m)^{r(a+b)}}{m^{a+b}}\right) \\[1.3ex]
	& \;\;\; = (-1)^{(r+1)(a+b)}\frac{1}{|X_m|} \sum_{\chi \in X_m} P^{(a,b)} \left( \Phi(\chi, r, m^2) \right) + O \left( \frac{(\log m)^{r(a+b)}}{m^{a+b}}\right) \\[1.3ex]
	& \;\;\; =  (-1)^{(r+1)(a+b)} \mu^{(a,b)}(r) + O\left( \frac{(\log m)^{(r+1)(a+b) + 2}}{m} \right).
\end{align*}
$\;$ \hfill $\qed$ \vspace{2mm}\\

To obtain an unconditional version of the above theorem, we need to take a closer look at estimating the $\mathfrak{Z}_{\chi}(r,x)$ term. We have
\begin{align} \label{zerosum2}
\nonumber		&  \left| \sum_{\rho} \left( \sum_{k=0}^{r} (-1)^{r+k} k! \binom{r}{k} \left[ \frac{1}{(\rho-1)^{k+1}} - \frac{1}{\rho^{k+1}} \right] \frac{ x^{\rho} (\log x)^{r-k}}{(x-1)}\right) \right| \\ \nonumber	
       \leq \; &  \frac{(r!)2^r (\log x)^{r}}{(x-1)} \sum_{\rho = \beta + i \gamma}  \sum_{k=0}^{r}  \left|  \frac{1}{(\rho-1)^{k+1}} - \frac{1}{\rho^{k+1}}   \right| x^{\beta}   \\
       \leq \; &  \frac{(r!)2^r (\log x)^{r}}{(x-1)} \sum_{\rho = \beta + i \gamma}  \sum_{k=0}^{r}  \left(  \frac{1}{|\rho-1|^{k+1}} + \frac{1}{|\rho|^{k+1}}   \right) x^{\beta}.
	\end{align}
    We now state several results used to estimate (\ref{zerosum2}). We will also prove several lemmas that depend on the behavior and estimates of zeros of $L(s, \chi)$ in the critical strip. To begin with, we will use the following two well-known results, due to Gronwall, Titchmarsh, Siegel etc.; see \cite{dave}, \S14, 16 and 21. \\

\textbf{Theorem (A)} There exists an absolute and effective positive constant $c$ such that if $\rho = \beta + i \gamma$ is a nontrivial zero of $L(s, \chi)$ with $|\gamma| \leq T$, $T \geq 1$, then either:
	$$ \text{min} (1- \beta, \beta) \; > \; \dfrac{c}{\log (mT)},$$
		or $\chi = \chi_1$ and either $\rho = \beta_1$ or $\rho = 1-\beta_1$, where $\beta_1$ is a real simple zero satisfying $ \beta_1 > \frac{1}{2} $ and $1 - \beta_1 \gg m^{-\e}$. \\

\textbf{Theorem (B)} 
	Let $Z_\chi$ be the set of nontrivial zeros of $L(s, \chi)$. Then 
	$$\# \{  \beta + i \gamma \in Z_\chi \; : \; |\gamma - T| < 1 \} \ll \log(m(T+2)).$$

\begin{lm}\label{sublem1}
	We have
\begin{equation}
\sideset{}{'}\sum_{|\gamma| \leq 1} \sum_{k=0}^{r}  \left(  \frac{1}{|\rho-1|^{k+1}} + \frac{1}{|\rho|^{k+1}}   \right) x^{\beta}  	\ll (r+1) x (\log m)^{r+2},
\end{equation}
where $ \; \sideset{}{'}\sum$ is the sum over all $\rho$ excluding the possible exceptional zero. \\
\end{lm}
\begin{proof}
		For  $|\gamma| \leq 1$, by Theorem (A) with $T=1$ and with $\rho$ not being exceptional, we see that $|\rho| \geq |\beta| \gg \frac{1}{\log m}$. Similarly, $|1 - \rho|  \gg \frac{1}{\log m}$. Hence,
	$$\frac{1}{| \rho |^{k+1}} \ll (\log m)^{k+1}, \;\;\;  \frac{1}{| 1 - \rho |^{k+1}} \ll (\log m)^{k+1}.$$
	Therefore,
	\begin{align*}
		& \sideset{}{'}\sum_{|\gamma| \leq 1} \sum_{k=0}^{r}  \left(  \frac{1}{|\rho-1|^{k+1}} + \frac{1}{|\rho|^{k+1}}   \right) x^{\beta}
	\; \ll (r+1) x(\log m)^{r+1}  \;\left( \sum_{|\gamma| \leq 1} {'} \; 1
		\; \right)   \ll \; (r+1) x (\log m)^{r+2}.
	\end{align*}
\end{proof}
\begin{lm}\label{sublem2} For $T \geq 1$, we have
	\begin{equation} 
			\sum_{|\gamma| > T } \sum_{k=0}^{r}  \left(  \frac{1}{|\rho-1|^{k+1}} + \frac{1}{|\rho|^{k+1}}   \right) x^{\beta}  	\ll \frac{(r+1) x (\log mT)}{T}.
	\end{equation}
\end{lm}
\begin{proof}
	Note that
\begin{align*}
  & \sum_{|\gamma| > T } \sum_{k=0}^{r}  \left(  \frac{1}{|\rho-1|^{k+1}} + \frac{1}{|\rho|^{k+1}}   \right) x^{\beta} 
 \;\;\; \ll (r+1) x \sum_{|\gamma| > T } \; \frac{1}{\gamma^2} \\[1.5ex]
  \ll \; & (r+1) x \; \sum_{j = [T]}^{\infty} \; \frac{1}{j^2} \sum_{|\gamma - (j+1)|<1  } \; 1 
 \;\;\; \ll (r+1) x \; \sum_{j = [T]}^{\infty} \; \frac{\log (m(j+3))}{j^2} \\[1.3ex]
 \ll \; & \frac{(r+1) x (\log mT)}{T} .
\end{align*}
\end{proof}
The following result is part of the proof of a sublemma (5.4.4) of \cite{ihara2}. We record it here as a lemma and, for completeness, also include the proof. 
\begin{lm}\label{sublem3}
	For $T \geq 2$ and $x \geq (mT)^6$, we have 
		\begin{equation} 
			\sum_{\chi \in X_m } \; \sideset{}{'} \sum_{\substack{ \rho \in Z_{\chi} \\ |\gamma | \leq T}} x^{\beta} \ll x (\log x)^{14}.
	\end{equation}
\end{lm}
\begin{proof}
	Let us denote 
		$$\tilde{S} (x,m,T) = 		\sum_{\chi \in X_m } \; \sideset{}{'} \sum_{\substack{ \rho \in Z_{\chi} \\ |\gamma | \leq T}} x^{\beta}.$$
	The lemma is a consequence of well-known bounds for the number $N(\sigma, T, m)$ related to the number of zeros of $L(s, \chi)$ in a rectangle. In particular, for $0 \leq \sigma \leq 1$ and $T \geq 2$, define 
	\begin{equation*}
		\begin{cases}
		N(\sigma, T, \chi) = \# \left\{ \rho = \beta+ i \gamma  \in Z_{\chi} \; : \; \beta \geq \sigma, \; |\gamma| \leq T \right \} \\[1.3ex]
			N( \sigma, T, m) = \sum \limits_{\chi \in X_m} N(\sigma, T, \chi).
		\end{cases}
	\end{equation*} 
It is well known that $N(0, T, \chi) \ll T \log(mT)$ (e.g. see \S 16 of \cite{dave}) and thus $N(0, T, m) \ll mT \log(mT)$. We will also use the following result of Montgomery \cite[Theorem 12.1]{mont}; see also \cite{mont1}: \\

For $\sigma \geq 4/5$ and $T \geq 2$, 
\begin{equation} \label{mont1}
	N(\sigma, T, m) \ll (mT)^{\frac{2(1- \sigma)}{\sigma}}(\log mT)^{14} \ll (mT)^{\frac{5}{2}(1- \sigma)} (\log mT)^{14}.
\end{equation}
A similar result can also be found in \cite{huxjut}. We rewrite $\tilde{S}(x, m, T)$ as 
\begin{align*}
		\tilde{S}(x, m, T) & = \sum_{\chi \in X_m } \; \sideset{}{'} \sum_{\substack{ \rho \in Z_{\chi} \\ |\gamma | \leq T \\ \beta < 4/5}} x^{\beta}  \; + \;  \sum_{\chi \in X_m } \; \sideset{}{'} \sum_{\substack{ \rho \in Z_{\chi} \\ |\gamma | \leq T \\ 4/5 \leq \beta < 1}} x^{\beta}.
\end{align*}
The first summand is $$ \ll x^{4/5} N(0, T, m) \ll x^{4/5} (mT) (\log mT) \ll x^{4/5 + 1/6} \log x \ll x,$$
where the last inequality follows from the imposed condition $x \geq (mT)^6$. The second summand is 
\begin{align*}
	& \leq \left| \int_{4/5}^{1} x^{\sigma} d_{\sigma} N(\sigma, T , m) \right| \leq x^{4/5}N(4/5, T, m)+ \left| \int_{4/5}^{1} (x^{\sigma} \log x)  N(\sigma, T , m) \; d \sigma \right|  \\[1.3ex]
		& \ll x^{4/5} (mT)^{1/2}(\log mT)^{14} + (\log x) (mT)^{5/2}(\log mT)^{14} \int_{4/5}^1 \left( \frac{x}{(mT)^{5/2}}\right)^{\sigma} \; d\sigma.
\end{align*}
Note that the first term is $\ll x$. For the integral, we have 

\begin{align*}
	\int_{4/5}^1 \left( \frac{x}{(mT)^{5/2}}\right)^{\sigma} \; d\sigma &= \left[  \frac{\left( \frac{x}{(mT)^{5/2}}\right)^{\sigma}}{ \log \left( \frac{x}{(mT)^{5/2}}\right)} \right]_{4/5}^1 \\[1.3ex] 
		& \ll  \frac{x}{(mT)^{5/2}  (\log x)},
\end{align*}
and so the second term is $ \ll x (\log mT)^{14} \ll x (\log x)^{14}$. Hence, the lemma is proved. 
\end{proof}

\begin{lm}\label{sublem4}
For $T > 1$ and $x \geq (mT)^6$ we have 
	\begin{equation} 
		\sum_{\chi \in X_m } \; \sideset{}{'} \sum_{\substack{ \rho \in Z_{\chi} \\ |\gamma | \leq T}} \sum_{k=0}^{r}  \left(  \frac{1}{|\rho-1|^{k+1}} + \frac{1}{|\rho|^{k+1}}   \right) x^{\beta} 	\ll (r+1) x (\log x)^{r+15}.
\end{equation}
\end{lm}
\begin{proof}
Keeping notation similar to that in \cite[5.6]{ihara2}, we denote
\begin{equation}
	S(x, r, m , T) = \sum_{\chi \in X_m } \; \sideset{}{'} \sum_{\substack{ \rho \in Z_{\chi} \\ |\gamma | \leq T}} \sum_{k=0}^{r}  \left(  \frac{1}{|\rho-1|^{k+1}} + \frac{1}{|\rho|^{k+1}}   \right) x^{\beta}.
\end{equation}
Note that for all $\rho$ with $\beta \leq \frac{4}{5}$, $S(x, m, T) \ll (r+1) x^{4/5} (\log mT)^{e} \ll (r+1) x$, for some sufficiently large power $e$ depending on $r$. This can be easily seen as a slight modification of Lemma \ref{sublem1}. We therefore focus on the zeros $\rho$ with $\beta \geq \frac{4}{5}$. In this case, as before, we divide the sum into the ranges $|\gamma| \leq 2$ and $2 < |\gamma| \leq T$. 
Since $\beta \geq \frac{4}{5}$, we have $$\text{min}(\beta, 1-\beta) = 1-\beta > \frac{c}{ \log(mT)}.$$ 
Thus, \begin{align*}
	&|\rho|^{k+1} > \beta^{k+1} > \left(\frac{4}{5}\right)^{k+1} \;\;\; \text{and} \\
	&|1-\rho|^{k+1} \geq (1-\beta)^{k+1} > \frac{c^{k+1}}{(\log mT)^{k+1}}.
\end{align*}
Thus, we have
\begin{equation} \label{sublemeqn1}
	S(x, r , m, T) \ll (r+1)(\log mT)^{r+1} \tilde{S}(x,m,2) + S_1(x, r, m,T),
\end{equation}
where
\begin{align}\label{sublemeqn2}
\nonumber S_1(x,r, m,T) &= \sum_{\chi \in X_m } \; \sideset{}{'} \sum_{\substack{ \rho \in Z_{\chi} \\ 2<|\gamma | \leq T}} \sum_{k=0}^{r}  \left(  \frac{1}{|\rho-1|^{k+1}} + \frac{1}{|\rho|^{k+1}}   \right) x^{\beta} \\ \nonumber
&\ll (r+1) \sum_{\chi \in X_m} \; \sideset{}{'} \sum_{2 < |\gamma| \leq T} \frac{x^\beta}{\gamma^2} \\[1.3ex]
\nonumber & \leq (r+1)  \sum_{\substack{j \geq 0 \\ 2^{j+1} \leq T}}  \frac{1}{4^j} \;  \sum_{\chi \in X_m} \; \sideset{}{'} \sum_{2^j < |\gamma| \leq 2^{j+1}} x^\beta \\[1.3ex] 
& \leq (r+1)  \sum_{\substack{j \geq 0 \\ 2^{j+1} \leq T}} \frac{\tilde{S}(x, m, 2^{j+1})}{4^j} \ll (r+1)  x(\log x)^{14}.
\end{align}
Since $x \geq (mT)^6$, we thus get, putting equation (\ref{sublemeqn1}) and (\ref{sublemeqn2}) together, 
\begin{equation*}
	S(x, r, m,T) \ll (r+1) x (\log x)^{r+15}.
\end{equation*}
\end{proof}
We will now put these lemmas together using the following elementary inequality. This was used in 6.8 of \cite{iharaM} and 5.3 of \cite{ihara2}. We include a short proof as well.
\begin{pr} For any $w, z \in \C$ we have
\begin{equation*}
	|P^{(a,b)} (z+w) - P^{(a,b)}(z)| \leq (a+b)|w| (|z| + |w|)^{a+b-1}.
\end{equation*}
\end{pr}
\begin{proof}
	First, note that for any $n \geq 1$, 
	\begin{align*}
		|(z+w)^n - z^n| &= \left| \binom{n}{1}z^{n-1}w + \cdots + \binom{n}{n} w^n \right| \\
		& \leq n |w| \left( \sum_{i=1}^{n} \binom{n-1}{i-1} |z|^{n-i} |w|^{i-1} \right) \\
		&= n |w| (|z| + |w|)^{n-1},
	\end{align*}
where the last inequality follows from $\binom{n}{i} \leq n \binom{n-1}{i-1}$ for $1 \leq i \leq n$. Thus, 

\begin{align*}
		& |P^{(a,b)} (z+w) - P^{(a,b)}(z)|  = |(z+w)^a (\overline{z+w})^b - z^a \overline{z}^b | \;\;\; \\
		& \;\;\; = |(z+w)^a \overline{(z+w)}^b - z^a (\overline{z+w})^b +  z^a (\overline{z+w})^b- z^a \overline{z}^b | \\
		& \;\;\; \leq |z+w|^b |(z+w)^a - z^a| + |z|^a |(\overline{z+w})^b - \overline{z}^b| \\
		& \;\;\; \leq a |w| (|z| + |w|)^{a+b-1} + b (|z| + |w|)^a |\overline{w}| ( |\overline{z}| + |\overline{w}|)^{b-1}\\
			& \;\;\; \leq (a+b) |w| (|z|+|w|)^{a+b-1}.
\end{align*}
\end{proof}

Choosing $z =  \mathcal{L}^{(r)}(1, \chi)$ and $w = \Phi(\chi, r, x) - \mathcal{L}^{(r)}(1, \chi)$ gives 
\begin{align}\label{PabL}
\nonumber \left| P^{(a,b)} (\mathcal{L}^{(r)}(1, \chi)) - P^{(a,b)}(  \Phi(\chi, r, x) ) \right| & \leq (a+b) \left|  \Phi(\chi, r, x) - \mathcal{L}^{(r)}(1, \chi) \right| \cdot \\ \nonumber 
& \left( \left| \Phi(\chi, r, x) - \mathcal{L}^{(r)}(1, \chi) \right| + \left|  \mathcal{L}^{(r)}(1, \chi) \right| \right)^{a+b-1}. \\
\end{align}
\vspace{1mm}\\
Let us denote the unique real quadratic character in $X_m$ by $\chi_1$. 
\begin{pr}\label{zerosum1} $\;$ 
	\begin{enumerate} 
	\item For $\chi \in X_m$, $x \geq m$ and $\epsilon >0$, we have  \begin{align*}
		& \left| P^{(a,b)} (\mathcal{L}^{(r)}(1, \chi)) - P^{(a,b)}(  \Phi(\chi, r, x) )  \right| \\
		& \;\;\;\;\;\; \ll \begin{cases}
		  ((r+1)!2^r (\log x)^{(r+1)} m^{(r+2)\epsilon})^{a+b} \text{ \hspace{2cm} for } \chi = \chi_1  \\ 
          \left((r+1)!2^r (\log x)^{r+1}(\log m)^{r+2} \right)^{(a+b-1)} \cdot \left|  \Phi(\chi, r, x) - \mathcal{L}^{(r)}(1, \chi) \right|\;\;\;\;\;\;  \text{ for } \chi \neq \chi_1
			\end{cases}.
	\end{align*}
\item For $x \geq m^{12}$, we have
\begin{equation}
	\sum_{\substack{\chi \in X_m \\ \chi \neq \chi_1}} \left|  \Phi(\chi, r, x) - \mathcal{L}^{(r)}(1, \chi) \right|  \ll (\log x)^{16}.
\end{equation}
	\end{enumerate}
\end{pr}
\begin{proof}

		(1)  By Lemmas \ref{sublem1} and \ref{sublem2} with $T=1$ we see that, for $\chi \neq \chi_1$, 
	$$E_1(r, x) \ll 2^r (r+1)! (\log x)^r(\log m)^{r+2},$$ and for $\chi = \chi_1$, 
	$$E_1(r, x) \ll 2^r (r+1)! (\log x)^r [ (\log m)^{r+2} + m^{(r+2) \epsilon} ] \ll 2^r (r+1)!(\log x)^r m^{(r+2)\epsilon}.$$
		This inequality follows from Theorem (A), stated before Lemma \ref{sublem1}. Substituting these into \eqref{mainarithformula}, we get (note that we are abusing the notation here, i.e. we are using the same $E_1(r,x)$, etc., with the understanding that now $K=\Q$.)
	\begin{equation}
		\left| \Phi(\chi, r, x) - \mathcal{L}^{(r)}(1, \chi) \right| \ll \begin{cases}
			2^r (r+1)! (\log x)^r(\log m)^{r+2} \;\;\; \text{ for } \chi \neq \chi_1\\ 
			2^r (r+1)!(\log x)^r m^{(r+2)\epsilon} \;\;\;\;\; \text{ for } \chi = \chi_1
			\end{cases}.
	\end{equation}
	By \cite[Remark 5.3]{ghosh1} we have $\Phi (\chi, r, x) \ll |\Phi_\Q(r,x)| \leq (\log x)^{r+1}$. Combining these bounds with \eqref{mainarithformula}, we get $\mathcal{L}^{(r)}(1,\chi) \ll (r+1)!2^r (\log x)^{r+1} (\log m)^{r+2} $ for $x \geq m$ and $\chi \neq \chi_1$. For $\chi = \chi_1$, we get $\mathcal{L}^{(r)}(1, \chi_1) \ll (r+1)!2^r (\log x)^r(\log x + m^{(r+2)\epsilon})$. Substituting these bounds into equation (\ref{PabL}), for $\chi \neq \chi_1$ we get
\begin{align*}
\left| P^{(a,b)} (\mathcal{L}^{(r)}(1, \chi)) - P^{(a,b)}(  \Phi(\chi, r, x) ) \right| & \ll \left((r+1)!2^r (\log x)^{r+1}(\log m)^{r+2} \right)^{(a+b-1)} \cdot \\
& \;\;\;\;\;\; \;\;\;\;\;\;  \;\;\;\;\;\; \;\;\;\;\;\; \left|  \Phi(\chi, r, x) - \mathcal{L}^{(r)}(1, \chi) \right|,
\end{align*} 
	and, for $\chi = \chi_1$, we get 
\begin{align*}
\left| P^{(a,b)} (\mathcal{L}^{(r)}(1, \chi)) - P^{(a,b)}(  \Phi(\chi, r, x) ) \right| & \ll ((r+1)!2^r (\log x)^{(r+1)} m^{(r+2)\epsilon})^{a+b}.
\end{align*} 

(2) Putting $T=m$ in Lemma \ref{sublem2} we obtain 
\begin{equation*}
			\sum_{\chi \in X_m} \frac{1}{x-1} \sum_{|\gamma| > T } \sum_{k=0}^{r}  \left(  \frac{1}{|\rho-1|^{k+1}} + \frac{1}{|\rho|^{k+1}}   \right) x^{\beta}  	\ll \frac{(r+1)!2^r (\log x)^r (\log m)}{m}.
\end{equation*}
On the other hand, for $T=m$, Lemma \ref{sublem4} gives, for $x \geq m^{12}$ 
\begin{equation*}
		\sum_{\chi \in X_m } \frac{1}{x-1} \; \sideset{}{'} \sum_{\substack{ \rho \in Z_{\chi} \\ |\gamma | \leq T}} \sum_{k=0}^{r}  \left(  \frac{1}{|\rho-1|^{k+1}} + \frac{1}{|\rho|^{k+1}}   \right) x^{\beta} 	\ll (r+1)! 2^r  (\log x)^{2r+15}.
\end{equation*}
Therefore, $	\sum_{\substack{\chi \in X_m \\ \chi \neq \chi_1}} \left|  \Phi(\chi, r, x) - \mathcal{L}^{(r)}(1, \chi) \right| \ll (r+1)! 2^r  (\log x)^{2r+15} $.
\end{proof}
\subsection{Proof of Theorem \ref{mainthm3}} $\;$ \\
Taking $x = m^{12}$ and applying Proposition \ref{zerosum1}, we get
\begin{align*}
   \sum_{\chi \in X_m } \left| P^{(a,b)} (\mathcal{L}^{(r)}(1, \chi)) - P^{(a,b)}(  \Phi(\chi, r, x) ) \right|  &  \ll ((r+1)!2^r)^{a+b}m^{\epsilon'}.
\end{align*}
Hence, we have
\begin{align*}
    \frac{1}{|X_m|} \sum_{\chi \in X_m } \; P^{(a,b)} ( \mathcal{L}^{(r)}(1, \chi)) & = \frac{1}{|X_m|} \sum_{\chi \in X_m } \; P^{(a,b)} (\Phi(\chi, r, x) ) +O_{r,a,b}(m^{\epsilon' - 1}) \\
    & = \mu^{(a,b)}(r) + O_{r,a,b}(m^{\epsilon' - 1}).
\end{align*}
$\;$ \hfill $\qed$

\section{Li Coefficients and Zero-free regions}\label{sec4}

Let $K$ be a number field and consider the completed Dedekind zeta function
\begin{equation}\label{completedzeta}
	\xi_K(s) = s(s-1) 2^{r_2}\left( \frac{\sqrt{|d_K|} }{2^{r_2} \pi^{n_K/2}}\right)^{s} \Gamma\left(\frac{s}{2}\right)^{r_1} \Gamma(s)^{r_2} \zeta_K(s) \; , 
\end{equation}
where  $[K : \Q] = n_K$. In \cite{licoeff}, Li introduced the sequence of numbers  
\begin{equation}
	\lambda_n (K) = \frac{1}{(n-1)!} \frac{d^n}{ds^n} \left(s^{n-1} \log \xi_K(s) \right) \Big|_{s=1} \text{ \hspace{1cm} for } n \geq 1, 
\end{equation}
now known as Li coefficients, and showed the following result. \\
\textbf{Theorem (Li's Criterion).} 
The generalized Riemann hypothesis for $\z_K(s)$ holds if and only if $\; \lambda_n (K) \geq 0$  for all $n \geq 1$. \\

Let $K$ be a number field and let $d_K$ denote the absolute value of the discriminant of $K$. A well-known result of Stark \cite[Lemma 3]{stark1} says that, for $K \neq \Q$,  $\z_K(s)$ has at most one zero in the region
\begin{equation} \label{starkzerofree}
	1 - \dfrac{1}{4 \log d_K} \leq \text{ Re}(s) \leq 1, \;\;\;\; |\text{ Im}(s)| \;  \leq \dfrac{1}{4 \log d_K}.
\end{equation}

Moreover, if this zero exists, it is necessarily real and simple. As pointed out by Murty in \cite{MurStark}, it is better to call $\beta_K$ a \emph{Stark zero} than a \emph{Siegel zero} as it caters to a region less wide than that of a conjectural location of a Siegel zero. More recently, Kadiri \cite[Corollary1.2]{habiba} improved Stark's result and showed \\
\textbf{Theorem (Kadiri).} \emph{ For $d_K$ sufficiently large, $\z_K(s)$ has at most one zero in the region:
\begin{equation} \label{kadirizerofree}
	1 - \dfrac{1}{2 \log d_K} \leq \text{Re}(s) \leq 1, \;\;\;\; |\text{Im}(s)| \;  \leq \dfrac{1}{2 \log d_K}.
\end{equation}
  This zero, if it exists, is real and simple.}\\
  
  Brown in \cite{brownli} made an effort towards proving an effective version of Li's Criterion, showing that non-negativity of the first few $\lambda_i$'s gives zero-free regions of a certain shape around $s=1$. In particular, just $\lambda_2(K) \geq 0$ should imply nonexistence of the exceptional zero in the Stark region (see \cite[Theorem 5]{brownli}). However, subsequent authors and the MathSciNet reviewer have noted several errors in Brown's paper that invalidate the published proofs of the main theorems. Initially the author had intended to rewrite and thereby verify one of the theorems and include it as an appendix. However, in the process we have improved Brown's result and therefore have recorded it as a theorem in this section.
  
  \begin{tm}\label{mainthm4}
  	For a number field $K$, let $\lambda_2(K)$ denote the second Li coefficient. For sufficiently large $d_K$, if  $\lambda_2(K)\geq 0$, then $\z_K(s)$ has no zeros in the region 
  	\begin{equation} \label{myregion}
  		1 - \dfrac{1}{c \log d_K} \leq \text{Re}(s) \leq 1, \;\;\;\; |\text{Im}(s)| \;  \leq \dfrac{1}{c \log d_K},
  	\end{equation}
  	for any constant $c > 3$. In particular, there is no exceptional zero in this region.
  \end{tm}
  \begin{rem}
  	We note that the region mentioned in \eqref{myregion} is larger than Stark's region, but smaller than Kadiri's \eqref{kadirizerofree}. 
  \end{rem}
  We will record a few lemmas before proving the theorem. 
   \begin{lm}\label{xiupperbound}
  	For real $\sigma > 1$, and $\rho$ varying over all nontrivial zeros of $\z_K(s)$, we have
  	$$\sum_{\rho} \dfrac{1}{\sigma - \rho} \leq \frac{1}{\sigma - 1} + \frac{1}{2} \log d_{K}.$$
  \end{lm}
  
  This is a standard result. In fact, one can show the stronger result $$\sum \frac{1}{\sigma - \rho} \leq \frac{1}{\sigma - 1} + \frac{1}{2} \log d_{K} + \frac{\zeta_K'(s)}{\zeta_K(s)},$$ see \cite[Lemma 6.2]{anupram}. \\
  The next lemma holds for a slightly larger region and is recorded as such. For $A \geq 2$, assume there is at most one real zero $\beta_K$ in the region
  \begin{equation}\label{genStark}
  	1 - \dfrac{1}{A \log d_K} \leq \Re(s) \leq 1, \text{ and  } | \Im(s)| \leq \frac{1}{A \log d_K}.
  \end{equation} 
  \begin{lm}\label{supp1}
  	For every nontrivial zero $\rho$ of $\z_K(s)$ such that $\rho \not\in \{\beta_K, 1-\beta_K\}$, we have
  	\begin{enumerate}[label=(\roman*)]
  		\item $ \displaystyle \sum_{\rho \not\in \{ \beta_K, 1- \beta_K\} } \left( \dfrac{1}{\rho} + \dfrac{1}{1-\rho} \right)  \leq (A+1)  \sum_{\rho \not\in \{ \beta_K, 1- \beta_K\} }\left( \dfrac{1}{x - \rho} + \dfrac{1}{x - (1-\rho)} \right),$\\[1.2ex]
  		\item $\displaystyle \sum_{\rho \not\in \{ \beta_K, 1- \beta_K\} }\left( \dfrac{1}{|\rho|^2} + \dfrac{1}{|1-\rho|^2} \right)  \leq A(A+1) \sum_{\rho \not\in \{ \beta_K, 1- \beta_K\} }\left( \dfrac{1}{x - \rho} + \dfrac{1}{x-(1-\rho)} \right) \log d_K,$ \\[1.2ex]
  	\end{enumerate}
  	where $x= 1 + \frac{1}{\log d_K}$. 
  \end{lm}
  \begin{proof}
  	(i) For any zero $\rho$ of $\xi_K(s)$, we rewrite $\rho = 1 - \frac{\beta + i \gamma}{\log d_K}$. If $x \geq 1$ is real, then we have 
  	\begin{align*}
  		\text{Re} & \left( \frac{1}{x - \rho} + \frac{1}{x - (1 - \rho)} \right) \\ 
  		&= \left(\frac{\beta + (x-1) \log d_K}{(\beta+(x-1)\log d_K)^2 + \gamma^2} + \frac{x \log d_K - \beta}{(x \log d_K - \beta)^2 + \gamma^2}\right) \log d_K.       
  	\end{align*}
  	Substituting $x= 1$, and $1 + \frac{1}{\log d_K}$ respectively, we get 
  	\begin{equation}
  		\text{Re} \left( \frac{1}{1 - \rho} + \frac{1}{ \rho} \right) = \left(\frac{\beta }{\beta^2 + \gamma^2} + \frac{ \log d_K - \beta}{( \log d_K - \beta)^2 + \gamma^2}\right) \log d_K.
  	\end{equation}
  	\begin{equation}
  		\text{Re} \left( \frac{1}{x - \rho} + \frac{1}{x - (1 - \rho)} \right) = \left(\frac{\beta + 1}{(\beta+1)^2 + \gamma^2} + \frac{1 + \log d_K - \beta}{(1+ \log d_K - \beta)^2 + \gamma^2}\right) \log d_K.
  	\end{equation}
  	Note that, for $\rho \neq \beta_K$, by \eqref{genStark} we have, $\beta \geq \frac{1}{A}$ or $|\gamma| \geq \frac{1}{A}$. \vspace{2mm}\\
  	\textbf{Case I: $\beta \geq \frac{1}{A}$.} We analyze the two summands separately. We have
  	\begin{align*}
  		\frac{\beta}{\beta^2 + \gamma^2} \frac{(\beta+1)^2 + \gamma^2}{\beta+1} & \leq \frac{\beta}{\beta+1} \left( \frac{\beta + 1}{\beta} \right)^2 = \frac{\beta+1}{\beta} \leq A+1, \\
  		\text{and so we have }\;  \frac{\beta}{\beta^2 + \gamma^2} & \leq (A+1) \frac{(\beta+1)^2 + \gamma^2}{\beta+1}. \text{ Similarly,}
  	\end{align*}
  	\begin{align*}
  		\frac{\log d_K - \beta}{(\log d_K - \beta)^2 + \gamma^2} \cdot \frac{(1+ \log d_K - \beta)^2 + \gamma^2}{1+\log d_K - \beta} & \leq 1 + \frac{1}{\log d_K - \beta}.
  	\end{align*}
  	Now note that, since Re$(1 - \bar{\rho})$ is also outside the region in \eqref{genStark}, we have, $\frac{\beta}{\log d_K} \leq 1 - \frac{1}{A \log d_K}$. That is, $(\log d_K - \beta) \geq A$ and so we have 
  	\begin{align*}
  		\frac{\log d_K - \beta}{(\log d_K - \beta)^2 + \gamma^2} \cdot \frac{(1+ \log d_K - \beta)^2 + \gamma^2}{1+\log d_K - \beta} & \leq 1 + \frac{1}{\log d_K - \beta}\leq A+1. 
  	\end{align*}
  	Combining these two cases, we have, for $\beta \geq \frac{1}{A}$, 
  	$$\text{Re} \left( \frac{1}{1 - \rho} + \frac{1}{ \rho} \right) \leq (A+1) \text{ Re} \left( \frac{1}{x - \rho} + \frac{1}{x - (1 - \rho)} \right).$$
  	\textbf{Case II: $\beta < \frac{1}{A}$ and $|\gamma| \geq \frac{1}{A}$.}
  	\begin{align*}
  		\frac{\beta}{\beta^2 + \gamma^2} \frac{(\beta+1)^2 + \gamma^2}{\beta+1} & \leq \frac{\beta}{\beta+1} \left( 1+ \frac{2\beta + 1}{\gamma^2} \right) \\
  		& \leq \left( 1 - \frac{1}{\beta + 1} \right) \left( 1 + (\frac{2}{A}+1) \cdot A^2 \right)\\ 
  		& \leq \frac{1}{(A+1)} \cdot (A+1)^2 = A+1. 
  	\end{align*}
  	Similarly we have, 
  	\begin{align*}
  		\frac{\log d_K - \beta}{(\log d_K - \beta)^2 + \gamma^2} &\cdot \frac{(1+ \log d_K - \beta)^2 + \gamma^2}{1+\log d_K - \beta}  \leq 1 + \frac{1}{\log d_K - \beta}  \\
  		&  \leq 1 + \frac{A}{A\log d_K - 1} \\
  		&  \leq 1 + \frac{A}{A\log 3 - 1} \leq 1+A.
  	\end{align*}
  	This follows from the fact that the smallest possible absolute value of the discriminant of a number field $K \neq \Q$ is $3$, attained by the quadratic field $\Q(\sqrt{-3})$. Also, note that the last inequality holds as $A \geq 2$ and so, in particular, $A \geq \frac{2}{\log 3} \sim 1.8205 $. This completes the proof of (i) in the lemma.\\
  	
  	(ii) We note that
  	\begin{equation*}
  		\frac{1}{|\rho|^2} + \frac{1}{|1-\rho|^2} = (\log d_K)^2 \left(\frac{1}{\beta^2 + \gamma^2} + \frac{1}{(\log d_K - \beta)^2 + \gamma^2} \right).
  	\end{equation*}
  	Like before, for $\beta \geq \frac{1}{A}$ we have,
  	\begin{align*}
  		\frac{1}{\beta^2 + \gamma^2} \frac{(\beta+1)^2 + \gamma^2}{\beta+1}  \leq \frac{1}{\beta+1} \left( \frac{\beta + 1}{\beta} \right)^2  & \leq \frac{A}{A+1} (A+1)^2 = A(A+1) \\   
  		\frac{1}{(\log d_K - \beta)^2 + \gamma^2} \cdot \frac{(1+ \log d_K - \beta)^2 + \gamma^2}{1+\log d_K - \beta} &\leq \frac{1}{1 + \log d_K - \beta}\left( 1 + \frac{1}{\log d_K - \beta}\right)^2 \\
  		&\leq A(A+1).
  	\end{align*}
  	For the case $\beta < \frac{1}{A}$ and $|\gamma| \geq \frac{1}{A}$, let
  	\[
  	F(\beta,\gamma)
  	=
  	\frac{1}{\beta^2+\gamma^2}
  	\frac{(\beta+1)^2+\gamma^2}{\beta+1}.
  	\]
  	First, expand the numerator:
  	\[
  	(\beta+1)^2+\gamma^2
  	=
  	(\beta^2+\gamma^2)+2\beta+1.
  	\]
  	Hence
  	\[
  	F(\beta,\gamma)
  	=
  	\frac1{\beta+1}
  	+
  	\frac{2\beta+1}{(\beta+1)(\beta^2+\gamma^2)}.
  	\]
  	
  	We will focus on the second summand. Since $\beta > 0$, we see that $\frac{2\beta+1}{\beta + 1}>0$. Thus, for fixed $\beta$, writing $|\gamma|=t$ we get
  	\[
  	F(\beta,t)
  	=
  	\frac1{\beta+1}
  	+
  	\frac{2\beta+1}{(\beta+1)(\beta^2+t^2)}.
  	\] 
  	Therefore 
  	\[
  	\frac{\partial F}{\partial t}
  	=
  	-\frac{2t(2\beta+1)}{(\beta+1)(\beta^2+t^2)^2}
  	<0.
  	\]
  	Thus $F$ is decreasing as $|\gamma|$ increases. Therefore, its maximum under the constraint
  	$	|\gamma|\ge\frac1A$
  	occurs at $	|\gamma|=\frac1A$. That is, 
  	\[
  	F(\beta,\gamma) \leq F(\beta, {A^{-1}}) = \frac{1}{\beta+1} \left[ 1 + \frac{2\beta+1}{(\beta^2+A^{-2})} \right] \leq 1 + \frac{2\beta+1}{(\beta^2+A^{-2})}  .
  	\] 	
  	We will show that $\frac{2\beta+1}{(\beta^2+A^{-2})}  \leq A(A+1)-1$. This is equivalent to 
  	\[
  	(A(A+1)-1) \beta^2 -2\beta  + (A(A+1)-1)A^{-2}-1 \geq 0.
  	\]
  	Simplifying we get,
  	\begin{equation}\label{lem_inter}
  		(A^2 + A -1) \beta^2 - 2 \beta + ( A^{-1} - A^{-2}) \geq 0.
  	\end{equation}
  	Call the quadratic on the left-hand side $Q(\beta)$. Since $Q$ has a positive leading coefficient, it is sufficient to check the minimum value. The minimum of $Q$ occurs at 
  	$$\beta^* = \frac{1}{(A^2 + A -1)}.$$
  	Thus, $
  	Q(\beta^*) = - \dfrac{1}{(A^2 + A -1)} + ( A^{-1} - A^{-2}) = \dfrac{A-1}{A^2} - \dfrac{1}{(A^2 + A -1)}. 
  	$
  	Note that,
  	\begin{align*}
  		&  \; \dfrac{A-1}{A^2} \geq \dfrac{1}{(A^2 + A -1)} \\
  		\Leftrightarrow & \;  (A-1)(A^2 + A -1) \geq A^2 \\
  		\Leftrightarrow & \;  A^3 - A^2 -2A  +1 \geq 0.
  	\end{align*}
  	The last inequality holds for any $A \geq 2$. Thus $Q(\beta^*) \geq 0$ and so, 
  	\[
  	F(\beta, \gamma) \leq 1 + 1 + \frac{2\beta+1}{(\beta^2+A^{-2})}  \leq 1 + A(A-1) -1 \leq A(A+1).
  	\]
  	The remaining term is easier. We have 
  	\begin{align*}
  		\frac{1}{(\log d_K - \beta)^2 + \gamma^2} \cdot \frac{(1+ \log d_K - \beta)^2 + \gamma^2}{1+\log d_K - \beta} &\leq \frac{1}{1 + \log d_K - \beta}\left( 1 + \frac{1}{\log d_K - \beta}\right)^2 \\
  		& \leq \dfrac{1 + \log d_K -\beta}{(\log d_K - \beta)^2} \\ 
  		&\leq \dfrac{1}{(\log d_K - \beta)^2} + \dfrac{1}{(\log d_K - \beta)}\\
  		& \leq \dfrac{A^2}{(A\log d_K - 1)^2} + \dfrac{A}{(A\log d_K - 1)} \\
  		& \leq \dfrac{A^2}{(A\log 3 - 1)^2} + \dfrac{A}{(A\log 3 - 1)} \leq A^2 + A.
  	\end{align*}
  	As before, the last inequality follows from the assumption $A \geq 2$, in particular, $A \geq \frac{2}{\log 3}$.
  	This completes the proof.
  	\end{proof}
  	
  	\textbf{Proof of Theorem \ref{mainthm4}.} Let  $\lambda_2(K) \geq 0$ and assume that the real exceptional zero $\beta_K$ of $\zeta_K(s)$ exists. We will derive a contradiction. 
  	 In \cite[Lemma 3.1]{licoeff}, Li showed that the following identity
  	\begin{equation*}
  		\lambda_n(K) = \sum_{\rho} \left( 1 - \left(1 - \dfrac{1}{\rho} \right)^n \right)
  	\end{equation*}
  	holds for every positive integer $n$, where summation is taken over all nontrivial zeros of $\zeta_K(s)$. By pairing up $\rho$ and $1-\rho$ we get  
  	\begin{equation}
  		2 \lambda_n(K) = \sum_{\rho}  2 - \left( \dfrac{\rho - 1}{\rho} \right)^n  -  \left( \dfrac{\rho}{\rho-1} \right)^n.
  	\end{equation}
  	In particular, we have 
  	\begin{align}\label{Li2main}
  		\nonumber 2 \lambda_2(K) & = \sum_{\rho}  2 - \left( \dfrac{\rho - 1}{\rho} \right)^2  -  \left( \dfrac{\rho}{\rho-1} \right)^2 \\ \nonumber
  		\nonumber &= \sum_{\rho}  2 - \left( 1 - \dfrac{1}{\rho} \right)^2 - \left( 1+ \dfrac{1}{\rho-1} \right)^2 \\
  		\nonumber  &= \sum_{\rho} 2 \left( \dfrac{1}{\rho} + \dfrac{1}{1 - \rho} \right) - \left( \dfrac{1}{\rho^2} + \dfrac{1}{(1 - \rho)^2} \right) \\
  		\nonumber  & =   2 \left( \dfrac{2}{\beta_K} + \dfrac{2}{1 - \beta_K}  - \dfrac{1}{\beta_K^2} - \dfrac{1}{(1 - \beta_K)^2} \right) + \\
  		& \; \; \; \;  \sum_{\rho \not\in \{\beta_K, 1-\beta_K\} } 2 \left( \dfrac{1}{\rho} + \dfrac{1}{1 - \rho} \right) - \left( \dfrac{1}{\rho^2} + \dfrac{1}{(1 - \rho)^2} \right). \\
  		\nonumber
  	\end{align} 
  	Substituting $x= 1 + \frac{1}{\log d_K}$ and $\rho = \beta_K = 1 - \frac{c}{\log d_K}$ we get,
  	\begin{align}\label{secondary1}
  		\nonumber \frac{1}{x - \rho} + \frac{1}{x - (1 - \rho)} &= \left( \frac{1}{1+c} + \frac{ 1}{\log d_K + 1 - c} \right) \log d_K \\
  		\nonumber &\geq \frac{A}{A+1} \log d_K + \frac{\log d_K}{\log d_K + 1} \\
  		\nonumber&\geq \frac{A}{A+1} \log d_K + \frac{\log 3}{\log 3 + 1} \\
  		&\geq \frac{A}{A+1} \log d_K + 0.5.
  	\end{align}
  	The penultimate inequality follows from the fact that $\frac{\log x}{\log x + 1}$ is increasing for $x>0$. The smallest value for $K \neq \Q$ is attained at $d_K=3$. Thus, combining this with Lemma \ref{xiupperbound} we have 
  	\begin{align}
  		\nonumber \sum_{\rho \not\in\{\beta_K, 1-\beta_K\} } \left( \frac{1}{x - \rho} + \frac{1}{x - (1 - \rho)} \right) &\leq \frac{2}{x-1} + \log d_K - 2 \left( \frac{A}{A+1} \log d_K + 0.5\right)\\
  		\nonumber & = \left(3 - \frac{2A}{A+1} \right) \log d_K - 1 \\
  		&= \left( \frac{A+3}{A+1} \right) \log d_K - 1.
  	\end{align}
  	Thus, by Lemma \ref{supp1} we have
  	\begin{equation}\label{mainpart2}
  		\sum_{\rho \not\in \{ \beta_K, 1- \beta_K\} } \left( \frac{1}{\rho} + \frac{1}{1-\rho}\right) \leq (A+3) \log d_K - (A+1).
  	\end{equation}
  	On the other hand,
  	\begin{align}\label{mainpart3}
  		\nonumber - \sum_{\rho \not\in\{\beta_K, 1-\beta_K\} } \left( \dfrac{1}{\rho^2} + \dfrac{1}{(1 - \rho)^2} \right) &\leq \left| \sum_{\rho \not\in\{\beta_K, 1-\beta_K\}} \left( \dfrac{1}{\rho^2} + \dfrac{1}{(1 - \rho)^2} \right) \right| \\
  		\nonumber & \leq \sum_{\rho \not\in\{\beta_K, 1-\beta_K\}} \frac{1}{|\rho|^2} + \frac{1}{|1-\rho|^2}\\
  		&\leq A(A+3) (\log d_K)^2  -A(A+1) \log d_K.\\ \nonumber
  	\end{align}
  	Since $\beta_K$ is real, $\frac{2}{\beta_K} - \frac{1}{\beta_K^2} \leq 1$. On the other hand, $$\frac{2}{1-\beta_K} - \frac{1}{(1-\beta_K)^2} \leq {2A \log d_K} - {A^2(\log d_K)^2}.$$  
  	We have 
  	\begin{equation}\label{mainpart1}
  		2 \left( \dfrac{2}{\beta_K} + \dfrac{2}{1 - \beta_K}  - \dfrac{1}{\beta_K^2} - \dfrac{1}{(1 - \beta_K)^2} \right) \leq 2 + 4A \log d_K - 2A^2 (\log d_K)^2.
  	\end{equation}
  	Combining the inequalities in (\ref{mainpart1}), (\ref{mainpart2}) and (\ref{mainpart3}) and substituting into (\ref{Li2main}), we get
  	\begin{align}\label{mainpart4}
  		\nonumber 2 \lambda_2 &\leq  2 + 4A \log d_K - 2A^2 (\log d_K)^2 + A(A+3) (\log d_K)^2  - A(A+1) \log d_K + \\
  		& \;\;\;\; (A+3) \log d_K - (A+1)\\
  		\nonumber &= (-A^2 + 3A) (\log d_K)^2  + (-A^2 + 4A + 3)\log d_K + (1 - A).
  	\end{align}
  	We see that for $A>3$, the right-hand side is a quadratic in $\log d_K$, with a negative leading term. Hence, for sufficiently large $d_K$, $\lambda_2$ becomes negative, giving us a contradiction!

\section{Application to cyclotomic extensions}
In this section we will apply Theorem \ref{mainthm3} to deduce positivity of $\lambda_2(K)$ for certain cyclotomic fields $K$. 
For prime $m$, let $K_m =\Q(\zeta_m)$. We have
\[
\zeta_{K_m}(s)
=
\zeta(s)
\prod_{\substack{\chi\bmod m\\ \chi\ne\chi_0}}
L(s,\chi).
\]
We also have $n_{K_m} = m-1, \; |d_{K_m}| = m^{m-2}, \; r_1=0$, and $r_2 = (m-1)/2 $. Thus we have
\begin{align*}
		\xi_{K_m}(s) &= s(s-1) 2^{\frac{(m-1)}{2}} (2\pi)^{\frac{-(m-1)s}{2}} m^{\frac{s(m-2)}{2}} \Gamma(s)^{\frac{m-1}{2}} \zeta_{K_m}(s) \\
		&= s(s-1) 2^{\frac{(m-1)}{2}} (2\pi)^{\frac{-(m-1)s}{2}} m^{\frac{s(m-2)}{2}} \Gamma(s)^{\frac{m-3}{2}} \Gamma\left(\frac{s}{2}\right) \Gamma\left(\frac{s+1}{2}\right)\frac{2^{s-1}}{\sqrt{\pi}} \zeta_{K_m}(s) \\
		&= \left(s(s-1)\pi^{-s/2}\Gamma\left(\frac{s}{2}\right) \zeta(s) \right) \left( 2^{\frac{(m-3)(1-s)}{2}} \pi^{\frac{-(m-2)s-1}{2}} m^{\frac{s(m-2)}{2}} \Gamma\left(\frac{s+1}{2}\right) \Gamma(s)^{\frac{m-3}{2}} \right) \\
		& \;\;\;\;\;\; \times \prod_{\substack{\chi\bmod m\\ \chi\ne\chi_0}}
		L(s,\chi).
		\end{align*}
Writing the completed Riemann zeta function as $\xi(s)$ and taking the logarithmic derivative we get
\begin{align*}
	\frac{\xi'_{K_m}(s)}{\xi_{K_m}(s)} &= \frac{\xi'(s)}{\xi(s)} + \sum_{\substack{\chi\bmod m\\ \chi\ne\chi_0}} \mathcal{L}(s,\chi) +
	\frac{m-2}{2}\log m - \frac{m-3}{2}\log 2 - \frac{m-2}{2}\log \pi \\
	&\;\;\;\;\; + \frac{1}{2}\psi\left(\frac{s+1}{2}\right)
	+ \frac{m-3}{2}\psi(s).
\end{align*}
Here $\psi(s)=\Gamma'(s)/\Gamma(s)$ denotes the digamma function.
Substituting $s=1$ gives
\begin{align*}
	\lambda_1(K_m) &= \frac{\xi'_{K_m}(1)}{\xi_{K_m}(1)} \\
	&= \lambda_1(\Q) + \sum_{\substack{\chi\bmod m\\ \chi\ne\chi_0}} \mathcal{L}(1,\chi)
	+ \frac{m-2}{2}\log m - \frac{m-3}{2}\log 2 - \frac{m-2}{2}\log \pi
	+ \frac{m-2}{2}\psi(1) \\
	&= \lambda_1(\Q) + \sum_{\substack{\chi\bmod m\\ \chi\ne\chi_0}} \mathcal{L}(1,\chi)
	+ \frac{m-2}{2}\log m - \frac{m-3}{2}\log 2 - \frac{m-2}{2}\log \pi
	- \frac{m-2}{2}\gamma.
\end{align*}
\subsection{Proof of Theorem \ref{mainthm5}}
We first look at the second Li coefficient for $K_m = \Q(\zeta_m)$. We have
\begin{align*}
	\lambda_{2}(K_m)
	&=
	2\frac{\xi_{K_m}'}{\xi_{K_m}}(1)
	+
	\left(\frac{\xi_{K_m}'}{\xi_{K_m}}\right)'(1)  \\
	&=  2 \lambda_1(K_m) +  \left(\frac{\xi'}{\xi}\right)'(1) + \sum_{\substack{\chi\bmod m\\ \chi\ne\chi_0}} \mathcal{L}'(1,\chi) + \frac{1}{4}\psi'(1) + \frac{m-3}{2} \psi'(1).
\end{align*}
We note that $2 \lambda_1(\Q) + \left(\frac{\xi'}{\xi}\right)'(1)  = \lambda_2(\Q)$ and $\psi'(1) = \dfrac{\pi^2}{6}$. Thus we have 
\begin{align*}
	\lambda_{2}(K_m) &=  \lambda_2(\Q)  + (m-2)\log m + \frac{2m-5}{4}\frac{\pi^2}{6} - \left[ (m-3) \log 2 + (m-2)( \log \pi + \gamma) \right] \\ 
	& \;\;\; + 2 \sum_{\substack{\chi\bmod m\\ \chi\ne\chi_0}} \mathcal{L}(1,\chi) + \sum_{\substack{\chi\bmod m\\ \chi\ne\chi_0}} \mathcal{L}'(1,\chi).
\end{align*}
Since $|X_m|=m-2$, dividing both sides by $|X_m|$ gives
\begin{align*}
	\frac{\lambda_2(K_m)}{m-2}
	&= \frac{\lambda_2(\Q)}{m-2}+\log m
	+ \frac{2m-5}{4(m-2)}\frac{\pi^2}{6}
	- \frac{m-3}{m-2}\log 2 - \log \pi - \gamma \\
	&\;\;\; + \frac{2}{|X_m|}\sum_{\chi\in X_m}\mathcal{L}(1,\chi)
	+ \frac{1}{|X_m|}\sum_{\chi\in X_m}\mathcal{L}'(1,\chi).
\end{align*}
By Theorem \ref{mainthm3}, applied with $(a,b)=(1,0)$ and with $r=0,1$, we have
\[
	\frac{1}{|X_m|}\sum_{\chi\in X_m}\mathcal{L}(1,\chi)=O_{\e}(m^{\e-1})
	\qquad \text{and} \qquad
	\frac{1}{|X_m|}\sum_{\chi\in X_m}\mathcal{L}'(1,\chi)=O_{\e}(m^{\e-1}),
\]
since $\mu^{(1,0)}(r)=0$ for these values of $r$. Hence
\begin{equation}\label{finaleq}
	\frac{\lambda_2(K_m)}{m-2}
	= \frac{\lambda_2(\Q)}{m-2}+\log m
	+ \frac{2m-5}{4(m-2)}\frac{\pi^2}{6}
	- \frac{m-3}{m-2}\log 2 - \log \pi - \gamma
	+ O_{\e}(m^{\e-1}).
\end{equation}
Therefore $\log m$ dominates the negative terms and the error in the right-hand side of \eqref{finaleq} and so, for $m$ sufficiently large, $\lambda_2(K_m) >0$. \qed

\begin{rem}
	Biane et al. gave a probabilistic interpretation in \cite[Section 2.3, p. 441]{biane} of the Riemann $\xi$ function as the Mellin transform of a probability law, namely $2\xi(s)=\mathbb E(X^s)$ for a suitable positive random variable $X$. In this interpretation, the relevant logarithmic derivatives of $\xi$ are cumulants of $\log X$, so in particular $\lambda_2(\mathbb Q)=2\lambda_1(\mathbb Q)+\operatorname{Var}(\log X)>0$. In \cite[Table 1, p. 533]{coffey} Coffey showed $\lambda_2(\Q) \sim 0.0923457$ through computational methods. 
\end{rem}
\subsection{Proof of Corollary \ref{maincor1}}
Since $\Q(\zeta_m)$ contains the quadratic field $\Q\!\left(\sqrt{(-1)^{(m-1)/2}m}\right)$, that is,
$
\Q(\zeta_m) \text{ contains }
	\Q(\sqrt m),  \text{ for }m\equiv 1\pmod 4, \text{ and }
	\Q(\sqrt{-m}), \text{ for }m\equiv 3\pmod 4.
$
For any Galois extension of number fields $L/K$, the Aramata-Brauer theorem tells us that $\zeta_L(s) / \zeta_K(s)$ is entire. Thus if $\zeta_{K_m}(s)$ is zero-free in the mentioned region, so is $\zeta_{\Q(\sqrt{m})}(s)$ or $\zeta_{\Q(\sqrt{-m})}(s)$, depending on the parity of $m$. \qed
\bibliographystyle{amsplain}
\bibliography{EulerKron}
\end{document}